\documentclass{amsart}
\usepackage{amsthm}
\usepackage{graphicx}

\newtheorem{Thm}{Theorem}[section]
\newtheorem{Cor}[Thm]{Corollary}

\theoremstyle{definition}
\newtheorem{Def}[Thm]{Definition}
\newtheorem{Exa}[Thm]{Example}
\theoremstyle{remark}
\newtheorem{Rem}[Thm]{Remark}
\numberwithin{equation}{section}

\DeclareMathOperator{\Ii}{i}
\DeclareMathOperator{\Dd}{d}

\DeclareMathOperator{\drho}{d\rho}
\DeclareMathOperator{\dzeta}{d\zeta}

\DeclareMathOperator{\JJ}{J}

\let\Bbb=\mathbb

\begin{document}

\title[Locally one-to-one harmonic functions with starlike analytic part]
{Locally one-to-one harmonic functions with starlike analytic part}

\author[I. Hotta]{Ikkei Hotta}
\thanks{The first author was supported by Grant-in-Aid for JSPS Fellows No. 252550 
and Grant-in-Aid for Young Scientists (B) No. 26800053}
\address{Department of Mathematics\\
Tokyo Institute of Technology\\
2-12-1 Ookayama, Meguro-ku\\
Tokyo 152-8551\\
Japan}
%\email{ikkeihotta@gmail.com}

\author[A. Michalski]{Andrzej Michalski}
\address{Department of Complex Analysis\\
The John Paul II Catholic University of Lublin\\ 
Ul. Konstantyn{\'o}w 1H\\
20-950 Lublin\\
Poland}
%\email{amichal@kul.lublin.pl}

%\thanks{}

\date{\today}

\begin{abstract}
%Clunie and Sheil-Small introduced in \cite{ClunieSheilSmall1} the family $S_H$
%of all normalized univalent harmonic mappings in the unit disc $\Delta$.
Let $L_H$ denote the set of all normalized locally one-to-one and sense-preserving harmonic functions in the unit disc $\Delta$.
It is well-known that every complex-valued harmonic function in the unit disc $\Delta$ 
can be uniquely represented as $f = h + \overline{g}$, where $h$ and $g$ are analytic in $\Delta$.
In particular the decomposition formula holds true for functions of the class $L_H$. 
For a fixed analytic function $h$, an interesting problem arises - to describe
all functions $g$, such that $f$ belongs to $L_H$.

The case when $f\in L_H$ and $h$ is the identity mapping was considered in \cite{KlimekMichalski3}. 
More general results are given in \cite{KlimekMichalski4}, 
where $f\in L_H$ and $h$ is a convex analytic mapping. 
The focus of our present research is to characterize the set of all functions $f\in L_H$ having starlike analytic part $h$.
In this paper, we provide coefficient, distortion and growth estimates of $g$.
We also give growth and Jacobian estimates of $f$.
\end{abstract}

\maketitle
\setcounter{page}{1}
\section{Introduction}
%Let $\Delta$ stands for the open unit disc in the complex plane $\Bbb C$. 
A complex-valued harmonic function $f$ in the open unit disc $\Delta\subset\Bbb C$ 
can be uniquely represented as
\begin{align}\label{I1}
f = h + \overline{g},
\end{align}
where $h$ and $g$ are analytic in $\Delta$ with $g(0)=0$.
Hence, $f$ is uniquely determined by coefficients of the following power series expansions
\begin{align}\label{I2}
h(z) = \sum_{n=0}^{\infty} a_n z^n,
\quad
g(z) = \sum_{n=1}^{\infty} b_n z^n,
\quad z\in\Delta,
\end{align}
where $a_n \in \Bbb C$, $n=0,1,2,\cdots$ and $b_n \in \Bbb C$, $n=1,2,3,\cdots$.

Such a function $f$, not identically constant,
is said to be sense-preserving in $\Delta$ if and only if it satisfies
the equation
\begin{align}\label{I3}
g' = \omega h',
\end{align}
where $\omega$ is analytic in $\Delta$ with $|\omega(z)|<1$ for all $z \in \Delta$.
The function $\omega$ is called the second complex dilatation of $f$, and it is closely related
to the Jacobian of $f$ defined as follows
\begin{align}\label{I5}
J_f(z):=|h'(z)|^2-|g'(z)|^2, \quad z \in \Delta.
\end{align}
Recall that a necessary and sufficient condition for $f$ to be locally one-to-one
and sense-preserving in $\Delta$ is $J_f(z)>0$, $z \in \Delta$.
This is an immediate consequence of Lewy's theorem (see \cite{Lewy1}).
Observe that if $J_f(z)>0$ then $|h'(z)|>0$ and hence $g'(z)/h'(z)$ is well defined for every $z \in \Delta$.
Thus the dilatation $\omega$ of a locally univalent and sense-preserving function $f$ in $\Delta$
can be expressed as
\begin{align}\label{I4}
\omega (z)= \frac{g'(z)}{h'(z)}, \quad z \in \Delta.
\end{align}

Let $L_H$ denote the set of all locally one-to-one and sense-preserving harmonic functions $f$ 
in $\Delta$ satisfying \eqref{I1}, such that $h(0)=0$ and $h'(0)=1$ ($g(0)=0$ by the uniqueness of \eqref{I1}).
Note that the known family $S_H$ introduced in \cite{ClunieSheilSmall1} by Clunie and Sheil-Small
is a subset of $L_H$, in fact, $S_H$ consists of all one-to-one functions in $L_H$.

The main idea of our research is to characterize subclasses of $L_H$ 
defined by some additional geometric conditions on $h$.
The case when $h$ is the identity mapping was studied in \cite{KlimekMichalski3}.
The paper \cite{KlimekMichalski4} was devoted to the case when $h$ 
is a convex analytic mapping. 
In this paper we consider functions $f\in L_H$ having starlike analytic part $h$.
\begin{Def}\label{envII4}
Let $\alpha \in [0,1)$. We define the class $\check{L}_H^{\alpha}$ of all $f \in L_H$,
such that $|b_1|=\alpha$ and $h$ is a starlike analytic function, where $b_1$ is taken
from the power series expansion \eqref{I2} of $f$.
Additionally, we define the class
\begin{align*}
\check{L}_H:=\bigcup\limits_{\alpha \in [0,1)}\check{L}_H^{\alpha}.
\end{align*}
\end{Def}
\begin{Rem}
Note that the estimate $|b_1|<1$ holds for all $f\in L_H$ (see \cite{Duren1}, p. 79), 
in particular, this explains why we have taken $\alpha \in [0,1)$ in Definition \ref{envII4}.
\end{Rem}
\begin{Exa}\label{envII5}
For a fixed $\zeta\in\Delta$ consider $f_{\zeta}:=k+\overline{g_{\zeta}}$ in $\Delta$, where
\begin{align*}
k(z):=\frac{z}{(1-z)^2}, \quad z\in\Delta
\end{align*}
and
\begin{align*}
g_{\zeta}(z):=\left(\frac{1-\zeta}{1+\zeta}\right)^2\log{\frac{1+\zeta z}{1-z}}
+\frac{z}{(1-z)^2}-\left(\frac{1-\zeta}{1+\zeta}\right)\frac{2z}{1-z}, \quad z\in\Delta.
\end{align*}
Straightforward calculation leads to the formula for the dilatation $\omega_{\zeta}$ of $f_{\zeta}$, i.e.
\begin{align*}
\omega_{\zeta}(z):=\frac{z+\zeta}{1+\zeta z}, \quad z\in\Delta,
\end{align*}
which ensures that $f_{\zeta}$ is locally one-to-one and sense-preserving in $\Delta$.
Clearly, $k(0)=0$, $k'(0)=1$ and $g_{\zeta}(0)=0$, hence $f_{\zeta}\in L_H$.
Finally, $k$ is well-known starlike function (Koebe function) and one can easily check that $g'_{\zeta}(0)=\zeta$,
therefore $f_{\zeta}\in \check{L}_H^{\alpha}$ with $\alpha:=|\zeta|$.
\end{Exa}
\section{Results}
At the beginning of this section we present a connection, discovered by us, between $\check{L}_H^{\alpha}$
and the class of normalized harmonic mappings with convex analytic part (introduced in \cite{{KlimekMichalski4}}).
It does not seem to be surprising in view of the classical result concerning convex and starlike analytic functions
due to J. W. Alexander.  
\begin{Thm}\label{envII1}
If $f\in\check{L}_H^{\alpha}$, then $F:=H+\overline{G}$ belongs to $S_H$, $H$ is a convex analytic function
and $|G'(0)|=\alpha$, where $H$ and $G$ satisfy the conditions $h(z)=zH'(z)$ and $g(z)=zG'(z)$, $z\in\Delta$ 
with the normalization $H(0)=0$ and $G(0)=0$.
\end{Thm}
\begin{proof}
Let $f\in\check{L}_H^{\alpha}$. By definition of $\check{L}_H^{\alpha}$, $h$ is a normalized starlike function. 
Hence, by Alexander's theorem, the function $H$ is a normalized convex function, namely, $H(0) = 0$ and $H'(0)=\lim_{z\to 0}h(z)/z=h'(0)=1$.
Moreover, $G'(0)=\lim_{z\to 0}g(z)/z=g'(0)$ and $|G'(0)|=\alpha$.
In particular, $|G'(0)|<|H'(0)|$.
Now, observe that for all $z\in \Delta\setminus\{0\}$ the interval $[0,h(z)]$ is a subset of $h(\Delta)$
since $h$ is a starlike function. Hence
\begin{align}\label{II1}
|g(z)|=\left|g\circ h^{-1}(h(z))\right|
\leq\int_0^{h(z)}\left|\frac{\Dd(g\circ h^{-1})}{\Dd\zeta}(\zeta)\right||\Dd\zeta|
<\int_0^{h(z)}|\Dd\zeta|=|h(z)|.
\end{align}
Again, by definition of $\check{L}_H^{\alpha}$, $f$ is a sense-preserving harmonic function which together with \eqref{II1} gives
\begin{align*}
|G'(z)|=\lim_{\zeta\to z}\left|\frac{g(\zeta)}{\zeta}\right|
<\lim_{\zeta\to z}\left|\frac{h(\zeta)}{\zeta}\right|=|H'(z)|,\quad z\in\Delta\setminus\{0\}.
\end{align*}
It shows that $F$ is a locally univalent and sense-preserving harmonic function in $\Delta$.
Finally, appealing to \cite[Corollary 2.3]{KlimekMichalski4}, we conclude that $F\in S_H$.
\end{proof}

For $f\in\check{L}_H^{\alpha}$, the classical theory of univalent functions says (see e.g. \cite{Duren2}, \cite{Goodman1})
\begin{align}\label{II13}
|a_n| \leq n, \quad n=2,3,4,\cdots\ .
\end{align}
The following theorem gives an estimate of the coefficient $b_n$.
\begin{Thm}\label{envII2}
If $f\in\check{L}_H^{\alpha}$, then we have
\begin{align}\label{II2}
|b_2|\leq 2\alpha+\frac{1-\alpha^2}{2}
\end{align}
and
\begin{align}\label{II3}
|b_n|\leq \alpha+\sqrt{(n-\alpha^2)(n-1)},\quad n=3,4,5,\cdots\ .
\end{align}
\end{Thm}
\begin{proof}
If $f\in\check{L}_H^{\alpha}$, then by Theorem \ref{II1} the function $F:=H+\overline{G}$ belongs to $S_H$, $H$ is a convex analytic function
and $|G'(0)|=\alpha$, where $H$ and $G$ satisfy the conditions $h(z)=zH'(z)$ and $g(z)=zG'(z)$ for all $z\in\Delta$ with the normalization $H(0)=0$ and $G(0)=0$.
Clearly $G$ can be expanded in a power series, say
\begin{align*}
G(z) = \sum_{n=1}^{\infty} B_n z^n,
\quad z\in\Delta,
\end{align*}
where $B_n \in \Bbb C$, $n=1,2,3,\cdots$.
Using this expansion together with the expansion \eqref{I2} of $g$
and the formula $g(z)=zG'(z)$, we obtain $b_n=nB_n$, $n=1,2,3,\cdots$.
Next, by applying \cite[Theorem 3.1]{KlimekMichalski4} we have
\begin{align*}
|b_n|\leq n\frac{\alpha+\sqrt{(n-\alpha^2)(n-1)}}{n},\quad n=2,3,4,\cdots, 
\end{align*}
which gives \eqref{II3}.
To improve the estimate in the case $n=2$, consider the function
\begin{align*}
F(z):=\frac{\omega(z)-\omega(0)}{1-\overline{\omega(0)}\omega(z)},\quad z\in\Delta, 
\end{align*}
where $\omega$ is the dilatation of $f$. Since $\omega$ is analytic in $\Delta$
it has a power series expansion, say
\begin{align*}
\omega(z) = \sum_{n=0}^{\infty} c_n z^n,
\quad z\in\Delta,
\end{align*}
where $c_n \in \Bbb C$, $n=0,1,2,\cdots$ and $|c_0|=|\omega(0)|=|g'(0)|=|b_1|=\alpha$.
Recall that $|\omega(z)|<1$ for all $z\in\Delta$. Hence we can apply the Schwarz lemma to $F$
and obtain $|F'(0)|\leq 1$, which yields 
\begin{align}\label{II12}
|c_1|=|\omega'(0)|\leq 1-|c_0|^2.
\end{align}
On the other hand, the formula \eqref{I4} gives  
\begin{align*}
2b_2=2a_2c_0+c_1,
\end{align*}
which together with \eqref{II13} and \eqref{II12} leads to the estimate
\begin{align*}
2|b_2|\leq 4|c_0|+1-|c_0|^2.
\end{align*}
Since this is equivalent to \eqref{II2}, the proof is completed.
\end{proof}
We have the following immediate corollary from Theorem \ref{envII2}.
\begin{Cor}\label{envII3}
If $f\in\check{L}_H$ then $|b_n| < n$, $n=2,3,4,\cdots$.
\end{Cor}

Recall that by definition the analytic part $h$ of $f\in\check{L}_H$ is starlike. Hence, it is known that
\begin{align}\label{II4}
\frac{1-|z|}{(1+|z|)^3}\leq |h'(z)|\leq \frac{1+|z|}{(1-|z|)^3}, \quad z\in\Delta.
\end{align}
Our next result is the distortion estimate of the anti-analytic part $g$ of $f\in\check{L}_H^{\alpha}$.
\begin{Thm}%\label{III22}
If $f \in \check{L}_H^{\alpha}$, then
\begin{align}\label{II5}
%|g'(z)| \geq \frac{(\alpha-|z|)(1-|z|)}{(1-\alpha |z|)(1+|z|)^3}, \quad z \in \Delta,
|g'(z)| \geq 
\begin{cases}
\displaystyle\frac{(\alpha-|z|)(1-|z|)}{(1-\alpha |z|)(1+|z|)^3},& |z|  < \alpha,\\[8pt]
0, & |z| \geq \alpha,
\end{cases}
\end{align}
where $z\in\Delta$ and
\begin{align}\label{II6}
|g'(z)| \leq \frac{(\alpha+|z|)(1+|z|)}{(1+\alpha |z|)(1-|z|)^3}, \quad z \in \Delta.
\end{align}
\end{Thm}
\begin{proof}
Let $\omega$ of the form \eqref{I4} be the dilatation of $f\in\check{L}_H^{\alpha}$
and let $g'(0)=\alpha e^{\Ii\varphi}$. Then the function
\begin{equation*}
\Omega(z)=\frac{ e^{-\Ii\varphi}\omega(z)-\alpha}{1-\alpha e^{-\Ii\varphi}\omega(z)}, \quad z\in\Delta
\end{equation*}
satisfies the assumptions of the Schwarz lemma, which gives
\begin{equation*}
\left| e^{-\Ii\varphi}\omega(z)-\alpha\right|\leq|z|\left|1-\alpha e^{-\Ii\varphi}\omega(z)\right|, \quad z\in\Delta.
\end{equation*}
Equivalently we can write
\begin{equation*}
\left| e^{-\Ii\varphi}\omega(z)-\frac{\alpha(1-|z|^2)}{1-\alpha^2|z|^2}\right|
\leq\frac{(1-\alpha^2)|z|}{1-\alpha^2|z|^2}, \quad z\in\Delta
\end{equation*}
and the equality holds only for the functions satisfying
\begin{align*}
\omega(z)=e^{i\varphi}\frac{e^{i\psi}z+\alpha}{1+\alpha e^{i\psi}z}, \quad z \in \Delta,
\end{align*}
where $\psi \in \Bbb R$.
Hence, by the triangle inequality we have
\begin{align}\label{II7}
\frac{\alpha-|z|}{1-\alpha|z|}\leq |\omega(z)|\leq\frac{\alpha+|z|}{1+\alpha|z|}, \quad z\in\Delta.
\end{align}
Finally, applying the estimate \eqref{II7} together with \eqref{II4} to the identity \eqref{I3}
we obtain \eqref{II5} and \eqref{II6}, so the proof is completed.
\end{proof}
\begin{Cor}
If $f\in\check{L}_H$, then
\begin{align*}
|g'(z)| \leq \frac{1+|z|}{(1-|z|)^3}, \quad z \in \Delta.
\end{align*}
\end{Cor}
Using the distortion estimates we can easily deduce the following Jacobian estimates of $f$.
\begin{Thm}
If $f \in \check{L}_H^{\alpha}$, then
\begin{align*}
\frac{(1-\alpha^2)(1-|z|)^3}{(1+\alpha|z|)^2(1+|z|)^5}\leq J_f(z)\leq
\begin{cases}
\displaystyle \frac{(1-\alpha^2)(1+|z|)^3}{(1-\alpha|z|)^2(1-|z|)^5},\quad &|z|<\alpha,\\[8pt]
\displaystyle \frac{(1+|z|)^2}{(1-|z|)^6},\quad &|z|\geq\alpha,
\end{cases}
\end{align*}
where $z\in\Delta$.
\end{Thm}
\begin{proof}
Observe that if $f\in \check{L}_H^{\alpha}$ then $h'$ does not vanish in $\Delta$
and we can write the Jacobian of $f$ given by \eqref{I5} in the form
\begin{align*}
\JJ_f(z)=|h'(z)|^2\left(1-|\omega(z)|^2\right), \quad z\in\Delta,
\end{align*}
where $\omega$ is the dilatation of $f$.
By applying \eqref{II4} and \eqref{II7} to the above formula we obtain
\begin{align*}
\frac{(1-\alpha^2)(1-|z|^2)}{(1+\alpha|z|)^2}\frac{(1-|z|)^2}{(1+|z|)^6}\leq J_f(z)\leq
\begin{cases}
\displaystyle \frac{(1-\alpha^2)(1-|z|^2)}{(1-\alpha|z|)^2}\frac{(1+|z|)^2}{(1-|z|)^6},\quad &|z|<\alpha,\\[8pt]
\displaystyle \frac{(1+|z|)^2}{(1-|z|)^6},\quad &|z|\geq\alpha
\end{cases}
\end{align*}
and the proof is completed.
Note that these estimates can also be deduced from a more general result given in \cite{KlimekMichalski5}.
\end{proof}

The growth estimate of the analytic part $h$ of $f\in\check{L}_H$ is known to be of the form
\begin{align}\label{II8}
\frac{|z|}{(1+|z|)^2}\leq |h(z)|\leq \frac{|z|}{(1-|z|)^2}, \quad z\in\Delta.
\end{align} 
In the following theorem we give the growth estimate of the anti-analytic part $g$.
\begin{Thm}%\label{III9}
If $f \in \check{L}_H^{\alpha}$, then
\begin{align}\label{II9}
%|g(z)| \geq \frac{|z|(\alpha-|z|)}{(1-\alpha |z|)(1+|z|)^2}, \quad z \in \Delta
|g(z)| \geq
\begin{cases}
\displaystyle\frac{|z|(\alpha-|z|)}{(1-\alpha |z|)(1+|z|)^2},& |z|  < \alpha,\\[8pt]
0, & |z| \geq \alpha,
\end{cases}
\end{align}
where $z\in\Delta$ and
\begin{align}\label{II10}
\begin{split}
|g(z)|
&\leq \left(\frac{1-\alpha}{1+\alpha}\right)^2\log{\frac{1+\alpha |z|}{1-|z|}}
+\frac{|z|}{(1-|z|)^2}-\left(\frac{1-\alpha}{1+\alpha}\right)\frac{2|z|}{1-|z|} \\
&\leq \frac{|z|(\alpha+|z|)}{(1+\alpha |z|)(1-|z|)^2}, \quad z \in \Delta.
\end{split}
\end{align}
\end{Thm}
\begin{proof}
If $f\in\check{L}_H^{\alpha}$, then by Theorem \ref{II1} the function $F:=H+\overline{G}$ belongs to $S_H$, $H$ is a convex analytic function
and $|G'(0)|=\alpha$, where $H$ and $G$ satisfy the conditions $h(z)=zH'(z)$ and $g(z)=zG'(z)$ for all $z\in\Delta$ with the normalization $H(0)=0$ 
and $G(0)=0$.
Hence, by applying \cite[Theorem 3.5]{KlimekMichalski4} together with $g(z)=zG'(z)$ we have
\begin{align*}
|z|\frac{\alpha-|z|}{(1-\alpha |z|)(1+|z|)^2} \leq |g(z)| \leq |z|\frac{\alpha+|z|}{(1+\alpha |z|)(1-|z|)^2}, \quad z \in \Delta.
\end{align*}
To prove the first inequality in \eqref{II10} we estimate the integral of $g'$ along $\gamma:=[0,z]$ by applying \eqref{II6}, i.e.
\begin{align*}
|g(z)|&=\left|\int_{\gamma}g'(\zeta)\dzeta\right|
\leq \int_{\gamma}|g'(\zeta)||\dzeta|
\leq \int_0^{|z|}\frac{(\alpha+\rho)(1+\rho)}{(1+\alpha\rho)(1-\rho)^3}\drho \\
&=\left(\frac{1-\alpha}{1+\alpha}\right)^2\log{\frac{1+\alpha |z|}{1-|z|}}
+\frac{|z|}{(1-|z|)^2}-\left(\frac{1-\alpha}{1+\alpha}\right)\frac{2|z|}{1-|z|}, \quad z \in \Delta.
\end{align*}
Finally, observe that the function $f_{\zeta}$ defined in Example \ref{envII5} 
with suitably chosen $\zeta\in\Delta$ shows that the first inequality in \eqref{II10} is best possible, 
which completes the proof.
\end{proof}

\begin{Cor}
If $f\in\check{L}_H$, then
\begin{align*}
|g(z)| \leq \frac{|z|}{(1-|z|)^2}, \quad z \in \Delta.
\end{align*}
\end{Cor}

Now, we can deduce the growth estimate of $f$.
\begin{Thm}%\label{III11}
If $f\in\check{L}_H^{\alpha}$, then
\begin{align}\label{II14}
|f(z)|\geq \frac{(1-\alpha)|z|(1-|z|)}{(1+\alpha|z|)(1+|z|)^2}, \quad z \in \Delta.
\end{align}
and
\begin{align}\label{II11}
\begin{split}
|f(z)|
&\leq \left(\frac{1-\alpha}{1+\alpha}\right)^2\log{\frac{1+\alpha |z|}{1-|z|}}
+\frac{2|z|}{(1-|z|)^2}-\left(\frac{1-\alpha}{1+\alpha}\right)\frac{2|z|}{1-|z|} \\
&\leq \frac{(1+\alpha)|z|(1+|z|)}{(1+\alpha |z|)(1-|z|)^2}, \quad z \in \Delta.
\end{split}
\end{align}
\end{Thm}
\begin{proof}
If $f\in\check{L}_H^{\alpha}$, then by Theorem \ref{II1} the function $F:=H+\overline{G}$ belongs to $S_H$, 
$H$ is a convex analytic function and $|G'(0)|=\alpha$, where $H$ and $G$ satisfy the conditions $h(z)=zH'(z)$ and $g(z)=zG'(z)$ 
for all $z\in\Delta$ with the normalization $H(0)=0$ and $G(0)=0$.
Hence, by applying inequality \eqref{II8} and \eqref{II7}, which also holds true for the dilatation of $F$ 
(see \cite[the proof of Theorem 3.5]{KlimekMichalski4}) we have
\begin{align*}
|f(z)|&=|h(z)+\overline{g}(z)| \geq |h(z)|-|g(z)|=|h(z)|\left(1-\left|\frac{g(z)}{h(z)}\right|\right)\\
&=|h(z)|\left(1-\left|\frac{G'(z)}{H'(z)}\right|\right)
\geq \frac{|z|}{(1+|z|)^2}\left(1-\frac{\alpha+|z|}{1+\alpha|z|}\right), \quad z \in \Delta.
\end{align*}
This proves the inequality \eqref{II14}.
To prove \eqref{II11} we use the triangle inequality $|f(z)| \leq |h(z)|+|g(z)|$,
which together with \eqref{II8} and \eqref{II10} leads to \eqref{II11} so the proof is completed.
\end{proof}

\begin{Cor}
If $f\in\check{L}_H$ then
\begin{align*}
|f(z)| \leq \frac{2|z|}{(1-|z|)^2}, \quad z \in \Delta.
\end{align*}
\end{Cor}

\begin{Rem}
The estimates \eqref{II3}, \eqref{II9} and \eqref{II14} are probably not precise.
\end{Rem}

\end{document}